\documentclass[a4paper,reqno,11pt]{amsart}
\usepackage{newtxtext}
\usepackage{microtype}
\usepackage{fix-cm}
\usepackage[T1]{fontenc}
\usepackage{natbib}
\usepackage{textcomp}
\usepackage{amsfonts} 
\usepackage{amssymb}
\usepackage{amsthm}
\usepackage{amsmath} 
\usepackage{mathrsfs} 
\usepackage{dsfont}
\usepackage{esint}
\usepackage{algorithm}
\usepackage{soul}
\usepackage{algpseudocode}
\usepackage[english]{babel} 
\usepackage{xcolor}
\definecolor{gray}{named}{gray}

\usepackage[left=3.5cm,right=3.5cm,top=3.5cm,bottom=3cm,headsep=0.7cm]{geometry}
\input{insbox}%%%%%%%%%%%%%% TeX macro,
\usepackage{tikz}
\usetikzlibrary{arrows}
 
\usepackage[scriptsize,bf]{caption}
\usepackage{enumerate}
\usepackage{enumitem}

\usepackage{hyperref}
\usepackage{cleveref}
\usepackage{mathtools}
\hypersetup{
    colorlinks,
    linkcolor={red!80!black},
    citecolor={blue!80!black}
}
\renewcommand{\div}{\mbox{\rm div}\;\!}
\theoremstyle{plain}
\begingroup
\theoremstyle{plain}
\newtheorem{theorem}{Theorem}[section]

\newtheorem{proposition}[theorem]{Proposition}

\newtheorem{lemma}[theorem]{Lemma}
\theoremstyle{definition}

\theoremstyle{remark}
\newtheorem{remark}[theorem]{Remark}

\endgroup

\theoremstyle{definition}
\theoremstyle{remark}

\numberwithin{equation}{section}

\makeatletter
\newcommand{\myitem}[1]{%
\item[#1]\protected@edef\@currentlabel{#1}%
}
\makeatother

\newcommand{\R}{\mathbb{R}}

\mathchardef\emptyset="001F

\author[C. Burtea]{Cosmin Burtea} 
\address[C. Burtea]{Université Paris Cité, Sorbonne Université, CNRS, IMJ-PRG, F-75013 Paris, France}
\email{cosmin.burtea@imj-prg.fr}

\author[T.~Crin-Barat]{Timoth\'ee~Crin-Barat$^*$} 
\address[T.~Crin-Barat]{Universit\'e Paul Sabatier,  Institut de Math\'ematiques de Toulouse, Route de Narbonne 118, 31062 Toulouse Cedex 9, France.}
\email{timothee.crin-barat@math.univ-toulouse.fr}

\author[P.~Gonin-{}-Joubert]{Pierre Gonin-{}-Joubert} 
\address[P.~Gonin-{}-Joubert]{Université Claude Bernard Lyon 1, Institut Camille Jordan, 21 Av. Claude Bernard, 69100 Villeurbanne, France.}
\email{goninjoubert@math.univ-lyon1.fr}

\keywords{Baer-Nunziato system, Pressure-relaxation, Kapila system, well-posedness.
\\\noindent$^*$Corresponding author: timothee.crin-barat@math.univ-toulouse.fr}
	
\title[]{On the relaxation towards mechanical equilibrium for two-pressure compressible flows}

\begin{document}

\begin{abstract}

We introduce a symmetrization of a one-velocity two-pressures Baer-Nunziato type model for mixtures of barotropic compressible fluids. It allows us to justify the zero compaction viscosity limit and to recover a solution of the so-called Kapila model.  On the other hand, the symmetrization highlights a pressure-induced stabilization mechanism which allows us to recover a global-in-time existence result for initial data close to constant states.
\end{abstract}
\maketitle

\section{Introduction}

\subsection{Introduction}

The Baer-Nunziato model \eqref{DBN} was introduced in \cite{Baer1986} to study deflagration to detonation transition in reactive granular media. In all generality it is a quasilinear system of PDEs of order one consisting of $5+2d$ equations which enjoy a rather non-standard mathematical structure: the equations are non-conservative and degenerate in the sense that the entropy naturally associated to the system is not strictly convex  \cite{Forestier2011,BurteaCrin-BaratTan2023,CordesseMassot2019}. In particular, the  Lax-Friedrichs theorem for symmetrization does not apply in this context. The \eqref{DBN} system of equations is obtained via volume averaging procedure of the single-phase balance equations \cite{truesdell1984thermodynamics,DrewPassman2006,IshiiTakashi2010} supplemented with a closure relation governing the evolution of the volume fraction in order to obtain a closed system. The resulting equations are non-conservative, due to interfacial exchanges between the two phases. It turns out that this model exhibits the appropriate algebraic structure governing nonhomogeneous continuous media interaction and is used to analyze multiphase flows, see for instance \cite{HerardSalehSeguin2018,Herard2005,CoquelHerardSalehSeguin2013} and the references cited within. 
\smallbreak
The study of the hyperbolicity of the \eqref{DBN} model was addressed in
\cite{Embid1992,Saurel1999,Saurel2003,gallouet2004numerical,Forestier2011}, symmetrization, which implies the local well-posedness of the model, at least in some regions of the phase space, was addressed in \cite{gallouet2004numerical,CoquelHerardSalehSeguin2013,SalehSeguin2020,HerardSalehSeguin2018,CordesseMassot2019}, while the Riemann problem was studied in \cite{Embid1992,Saurel1999,Saurel2003,gallouet2004numerical}. To our knowledge, the problem of symmetrizing a non-conservative system with degenerate entropy has no answer although a discussion is provided in \cite{CordesseMassot2019} for the case of non-conservative systems with strictly convex entropy.

In this kind of model, the two phases of the mixture are out of mechanical, thermal, and kinetic equilibrium, respectively. This leads to a system with stiff interaction terms and, for computational purposes, reduced models were derived to describe the flow in regions far from shock waves. This results in the so-called Kapila model, also known as the one-velocity, one-pressure two-temperatures model. 
\smallbreak
In this paper, we study from a mathematical point of view the one-dimensional ($d=1$) two-pressures one-velocity Baer-Nunziato model for barotropic compressible flows which reads
\begin{equation}
\left\{
\begin{array}
[c]{l}%
\partial_{t}\alpha_{\pm}+u\partial_x\alpha_{\pm}=\pm\dfrac{\alpha_{+}%
\alpha_{-}}{\mu}\left(  p_{+}\left(  \rho_{+}\right)  -p_{-}\left(
\rho_{-}\right)  \right)  ,\\
\partial_{t}\left(  \alpha_{\pm}\rho_{\pm}\right)  +\partial_x\left(
\alpha_{\pm}\rho_{\pm}u\right)  =0,\\
\partial_{t}(\rho u)+\partial_x(\rho u^2)+\partial_x p=0,
\end{array}
\right.  \tag{$BN$}\label{DBN}%
\end{equation}
where the characteristic state function of the two phases are denoted separately by $+$ and $-$, and the unknowns of the two phases are: 
\begin{itemize}
\item the densities: $\rho_{+},\rho_{-},$

\item the velocity of the mixture: $u\in \mathbb{R}$,

\item the volume fractions: $\alpha_{+}$ and $\alpha_{-}=1-\alpha_{+},$

\item the pressures: $p_{+}=p_{+}\left(  \rho_{+}\right)$ and $p_{-}=p_{-}\left(
\rho_{-}\right)  ,$ such that $p_{\pm}:\mathbb{R_+\rightarrow R_+}$ are given regular and
strictly increasing functions,

\item the mixture density $\rho=\alpha_{+}\rho_{+}+\alpha_{-}\rho_{-},$

\item the mixture pressure $p=\alpha_+ p_{+}\left(  \rho_{+}\right)  +\alpha_{-}p_{-}\left(  \rho
_{-}\right),$

\item the compaction viscosity $\mu>0$. 
\end{itemize}

Several quantities of interest do not appear explicitely in \eqref{DBN} but will be used in the following:

\begin{itemize}
\item the mass fractions: $:y_{\pm}=\frac{\alpha_{\pm}\rho_{\pm}%
}{\rho};$

\item the sound speeds: $c_{\pm}^{2}=p_{\pm}^{\prime}\left(
\rho_{\pm}\right)$,

\item the average sound speed: $\rho c^2=\alpha_+\rho_+ c_+^2+\alpha_-\rho_- c_-^2$.
\end{itemize}
\smallbreak
In \cite{Kapila2001}, a reduced model for System \eqref{DBN} was obtained by means of an informal asymptotic analysis in the zero compaction viscosity limit $\mu\to0$. Formally, when $\mu$ tends to $0$, solutions of System $\eqref{DBN}$ are expected to converge to solutions of the Kapila model:
\begin{equation}
\left\{
\begin{array}
[c]{l}%
\alpha_{+}+\alpha_{-}=1,\\
\partial_{t}\left(  \alpha_{\pm}\rho_{\pm}\right)  +\partial_x\left(
\alpha_{\pm}\rho_{\pm}u\right)  =0,\\
\partial_{t}(\rho u)+\partial_x(\rho u^2)+\partial_x p=0,\\
p=p_{+}\left(  \rho_{+}\right)  =p_{-}\left(  \rho_{-}\right).
\end{array}
\right.  \tag{$K$}\label{Kapila_chap_sym}%
\end{equation}

% In this paper we adress this problem from a mathematical point of view : we show that system \eqref{DBN}  admits classical solutions in Sobolev spaces on a time interval $T$ independent of
% $\mu$ along with uniform estimate. When $\mu\rightarrow0$ we recover a classical solution of the the Kapila model \eqref{K}. 

\subsection{Aims of the paper}
In this work, we revisit and refine the analysis introduced in \cite{BurteaCrin-BaratTan2023}, which was previously restricted to the power-law pressure laws
\(p_\pm(\rho_\pm)=A_\pm\rho_\pm^{\gamma_\pm}\).
We propose a symmetrization of System \eqref{DBN} that makes the pressure--dissipation structure of the system explicit.

This reformulation of the system allows us to establish two types of results. 
On the one hand, we prove the existence of classical solutions on a time interval that is uniform with respect to~$\mu$, and show that, in the limit $\mu \to 0$, these solutions converge to a classical solution of the so-called Kapila model. 
On the other hand, for fixed~$\mu$, we exploit the pressure--dissipation mechanism to establish the global-in-time existence of classical solutions to the system \eqref{DBN} for initial data sufficiently small and close to the pressure equilibrium.

%This structure enables us to justify the pressure-relaxation limit $\mu\to 0$ in a ill-prepared scenario. Moreover, within a perturbative framework, to show that the pressure-stabilization mechanism induces a stabilizing effect on the velocity and the density.

% \subsection{Aims of the paper}
% Our objective are twofold:
% \begin{itemize}
%     \item Local-in-time pressure-relaxation limit
%     \item Global-in-time existence by pressure-relaxation mechanism. (non-uniform in $\mu$).
% \end{itemize}
% To obtain both these resultn deriving suitable symmetrization, echoing the normal forms of Kawashima and Yong, is  crucial. This is the purpose of the next section.

\section{Reformulation of the system} \label{sec:sym}
\subsection{Symmetrization}

We first discuss the limiting case $\mu=0$ which corresponds to the Kapila System.
System \eqref{Kapila_chap_sym} can be symmetrized by observing that
\[
\alpha_{\pm}(D_{t}p_{\pm}+\rho_{\pm}p_{\pm}^{\prime}\partial_{x}u)+\rho_{\pm
}p_{\pm}^{\prime}D_{t}\alpha_{\pm}=0,
\]
which gives, after dividing by $\rho_{\pm}p_{\pm}^{\prime}$, summation and
taking into account that $D_{t}p_{+}=D_{t}p_{-}$,
\[
\left(  \frac{\alpha_{+}}{\rho_{+}p_{+}^{\prime}}+\frac{\alpha_{-}}{\rho
_{-}p_{-}^{\prime}}\right)  D_{t}p+\partial_{x}u=0.
\]
Above, and in the rest of this paper we use the notation
\begin{equation}
    \label{lag_deriv}
    D_t:=\partial_t+u\partial_x,
\end{equation}
for the material derivative.
Introducing the so-called \textit{Wood sound speed}
\[
\frac{1}{\rho c_{w}^{2}}=\frac{\alpha_{+}}{\rho_{+}p_{+}^{\prime}\left(
\rho_{+}\right)  }+\frac{\alpha_{-}}{\rho_{-}p_{-}^{\prime}\left(  \rho
_{-}\right)  }:=\frac{\alpha_{+}}{\rho_{+}c_{+}^{2}}+\frac{\alpha_{-}}%
{\rho_{-}c_{-}^{2}}%
\]
we obtain%
\[
D_{t}p+\rho c_{w}^{2}\partial_{x}u=0.
\]
Gathering the equations of the new set of variables, we obtain the \textit{Wood sound speed} symmetrization, see \cite{Gavrilyuk2011}, of the Kapila model \eqref{Kapila_sym}:
\begin{equation}
\left\{
\begin{array}
[c]{l}%
D_{t}y=0,\\
D_{t}p+\rho c_{w}^{2}\partial_{x}u=0,\\
D_{t}u+\frac{1}{\rho}\partial_{x}p=0,
\end{array}
\right.  \label{Kapila_sym}%
\end{equation}
where $y:=y_+$.
\medbreak
In order to symmetrize \eqref{DBN} we follow
two guiding principles:

\begin{itemize}
    \item  First, we would like to find a symmetrization that degenerates
into \eqref{Kapila_sym} when $\mu\rightarrow
0.$ The new set of variables is therefore sought in the form $V=\left(
y,{\pi},\delta {\pi},u\right)  $ where
\[%
\begin{array}
[c]{l}
p={\pi}+\delta {\pi}
\\
{\pi}={\pi}\left(  \alpha_{+},\rho_{+},\rho_{-}\right)  ,\\
\delta {\pi}=\delta {\pi}\left(  \alpha_{+},\rho_{+},\rho_{-}\right)
=R\left(  \alpha_{+},\rho_{+},\rho_{-}\right)  \left(  p_{+}\left(  \rho
_{+}\right)  -p_{-}\left(  \rho_{-}\right)  \right) .
\end{array}
\]
\item Secondly, in the equation of ${\pi}$,
the terms of order zero resulting from the change of variables should be at
least proportional to $\frac{1}{\mu}(\delta {\pi})^{2}$.

\end{itemize}

We claim that the two-pressures-one-velocity \eqref{DBN} model can be put under the form
\begin{equation}
\left\{
\begin{array}
[c]{l}%
D_{t}y=0,\\
D_{t}\delta {\pi}+\left(  \rho c^{2}-\rho c_{w}^{2}+f_{11}\delta
{\pi}\right)  \partial_x u+\left(  \dfrac{\rho
_{+}c_{+}^{2}}{\alpha_{+}}+\dfrac{\rho_{-}c_{-}^{2}}{\alpha_{-}}+f_{12}\delta
{\pi}\right)  \dfrac{\alpha_{+}\alpha_{-}\delta {\pi}}{\mu}=0,\\
D_{t}{\pi}+(\rho c_{w}^{2}+f_{21}\delta {\pi})\partial_x u+f_{22}\dfrac{\left(  \delta {\pi}\right)  ^{2}}{\mu}=0,\\
D_{t}u+\frac{1}{\rho}\partial_x {\pi}+\frac{1}{\rho}\partial_x\delta {\pi}=0.
\end{array}
\right.  \label{form}%
\end{equation}
In \eqref{form}, ${\pi}$ and $\delta {\pi}$ are such that $\delta
{\pi}$ is proportional to the pressure difference $p_{+}-p_{-}$ while
${\pi}+\delta {\pi}=p$. The terms $f_{ij}$ are functions of the unknowns which are independent of $\mu$. Therefore, system \eqref{form} is symmetrizable
with the help of a diagonal matrix: the diagonal terms are strictly positive
as soon as the pressure difference is sufficiently small (the smallness will
not be linked to $\mu$) and $\rho_+c_+^2\neq \rho_- c_-^2$. The algebraic structure of the system (symmetry and
also the particular structure of the terms of order zero) allows us to obtain estimates in Sobolev spaces which are
uniform w.r.t. the parameter $\mu$. This allows us to justify that classical solutions of
\eqref{DBN} tend to classical solutions of \eqref{Kapila_chap_sym}. The precise
mathematical results are stated below in Section \ref{sec:mainresults}. 
\smallbreak
There are at least two ways to reformulate \eqref{DBN} into \eqref{form}: the first uses an idea
introduced in \cite{BurteaCrin-BaratTan2023} and the second one, up to our knowledge,
seems to be new. The latter has the advantage of having $f_{22}=0$ which
allows to obtain a better result in terms of quantifying the error between the
two systems. 

\subsubsection{A first symetrisation}

\noindent\newline

The first approach adopted in \cite{BurteaCrin-BaratTan2023} can be summarized as follows.
We write the equation satisfied by the average pressure
\begin{equation}
D_{t}p+\left(  \alpha_{+}\rho_{+}c_{+}^{2}+\alpha_{-}\rho_{-}c_{-}^{2}\right)
\partial_{x}u+\left(  \rho_{+}c_{+}^{2}-\rho_{-}c_{-}^{2}\right)  \frac
{\alpha_{+}\alpha_{-}}{\mu}\delta p =\frac{\alpha_{+}\alpha_{-}}{\mu}(\delta
p)^{2}\label{ecuatia_p_manip}%
\end{equation}
and we compare it with the equation verified by the pressure difference
$\delta p=p_{+}-p_{-}$:
\begin{equation}
D_{t}\delta p+(\rho_{+}c_{+}^{2}-\rho_{-}c_{-}^{2})\partial_{x}u+\left(
\frac{\rho_{+}c_{+}^{2}}{\alpha_{+}}+\frac{\rho_{-}c_{-}^{2}}{\alpha_{-}%
}\right)  \frac{\alpha_{+}\alpha_{-}}{\mu}\delta
p=0.\label{ecuatia_deltap_manip}%
\end{equation}
Note that the variable $p$ is not a good choice of unknown since its equation contains terms that are linear in~$\delta p/\mu$.
Thus, we search for $R=R\left(  \alpha_{+},\rho_{+},\rho_{-}\right)$ such that
$p-R\delta p$ contains terms that are quadratic with respect to $\delta
p/\sqrt{\mu}$. A natural choice of variables is
\begin{equation}
\left\{
\begin{array}
[c]{l}%
\delta {\pi}:=\frac{\left(  \rho_{+}c_{+}^{2}-\rho_{-}c_{-}^{2}\right)
}{\left(  \frac{\rho_{+}c_{+}^{2}}{\alpha_{+}}+\frac{\rho_{-}c_{-}^{2}}%
{\alpha_{-}}\right)  }\delta p,\\
{\pi}:=p-\frac{\left(  \rho_{+}c_{+}^{2}-\rho_{-}c_{-}^{2}\right)  }{\left(
\frac{\rho_{+}c_{+}^{2}}{\alpha_{+}}+\frac{\rho_{-}c_{-}^{2}}{\alpha_{-}%
}\right)  }\delta p=\frac{\frac{\rho_{-}c_{-}^{2}}{\alpha_{-}}}{\frac{\rho
_{+}c_{+}^{2}}{\alpha_{+}}+\frac{\rho_{-}c_{-}^{2}}{\alpha_{-}}}p_{+}%
+\frac{\frac{\rho_{+}c_{+}^{2}}{\alpha_{-}}}{\frac{\rho_{+}c_{+}^{2}}%
{\alpha_{+}}+\frac{\rho_{-}c_{-}^{2}}{\alpha_{-}}}p_{-}.
\end{array}
\right.  \label{variable_new_old}%
\end{equation}
A lengthy yet straightforward computation, that resembles what was
presented in \cite{BurteaCrin-BaratTan2023}, shows that the equations verified by
$\left(  {\pi},\delta {\pi}, y,u\right)$, are of the form $\left(  \text{\ref{form}}\right)$.
In fact, it turns out that some simplifications can be made in the case of Gamma-law pressures $p_\pm(\rho_\pm)=c_\pm\rho_\pm^{\gamma_\pm}$ as we have
\[
\rho_{\pm}c_{\pm}^{2}=\rho_{\pm}p_{\pm}^{\prime}=\gamma_{\pm}p_{\pm}.
\]
In this case, if $\gamma_+\ne\gamma_-$, we can write the term proportional with $\dfrac{\delta p}{\mu}$
in equation $\left(  \text{\ref{ecuatia_p_manip}}\right)  $ as
\begin{align*}
\left(  \rho_{+}c_{+}^{2}-\rho_{-}c_{-}^{2}\right)  \frac{\alpha_{+}\alpha
_{-}}{\mu}\delta p  & =\left(  \gamma_{+}p_{+}-\gamma_{-}p_{-}\right)
\frac{\alpha_{+}\alpha_{-}}{\mu}\delta p\\
& =(\gamma_{+}-\gamma_{-})p\frac{\alpha_{+}\alpha_{-}}{\mu}\delta p+\left(
\gamma_{+}\alpha_{-}-\gamma_{-}\alpha_{+}\right)  \frac{\alpha_{+}\alpha_{-}%
}{\mu}(\delta p)^{2}%
\end{align*}
and we have
\[
\left(  \frac{\rho_{+}c_{+}^{2}}{\alpha_{+}}+\frac{\rho_{-}c_{-}^{2}}%
{\alpha_{-}}\right)  \frac{\alpha_{+}\alpha_{-}}{\mu}\delta p=\left(
\frac{\gamma_{+}}{\alpha_{+}}+\frac{\gamma_{-}}{\alpha_{-}}\right)
p\frac{\alpha_{+}\alpha_{-}}{\mu}\delta p+\left(
\gamma_{+}\alpha_{-}^2-\gamma_{-}\alpha_{+}^2\right)
\frac{(\delta p)^2}{\mu}.
\] This explains the set of variables which was used in
\cite{BurteaCrin-BaratTan2023} which is
\[
p-\dfrac{(\gamma_{+}-\gamma_{-})}{\dfrac{\gamma_{+}}{\alpha_{+}}+\dfrac{\gamma_{-}}{\alpha_{-}}}\delta p:=\dfrac{\dfrac{\gamma_{-}}{\alpha_-}p_++\dfrac{\gamma_{+}}{\alpha_+}p_-}{\dfrac{\gamma_{+}}{\alpha_{+}}+\dfrac{\gamma_{-}}{\alpha_{-}}}.
\]
This change of variable allows us to justify the limit $\mu\rightarrow 0$ and to obtain a convergence rate for a class of well-prepared initial data.

\bigskip

\subsubsection{A second symmetrization}

\noindent \newline

Another possibility is to search for the effective pressure in the particular
form ${\pi}={\pi}\left(  \rho,y\right)  $. This can be done by defining
the function ${\pi}$ implicitly from the relation:
\begin{equation}
\frac{1}{\rho}=\frac{y}{p_{+}^{-1}\left(  {\pi}\right)  }+\frac{1-y}%
{p_{-}^{-1}\left(  {\pi}\right)  }.\label{relatie}%
\end{equation}
The existence of $\pi=\pi(\rho,y)$ is a consequence of the strict monotonicity of $p_\pm$ which guarantees that given any $(\rho,y)$ there exists an unique $\pi$ verifying \eqref{relatie} while its regularity follows from the implicit function theorem. We introduce the functions $\left[  \rho_{\pm}^{2}c_{\pm}^{2}\right] : (0,\infty)\rightarrow(0,\infty),\left[
\rho^{2}c_{w}^{2}\right]  :\left(  0,1\right)  \times\left(  0,\infty\right)
\rightarrow\left(  0,\infty\right)  $ defined by%
\[
\left\{
\begin{array}
[c]{l}%
\left[  \rho_{\pm}^{2}c_{\pm}^{2}\right]  \left(  s\right)  =(p_{\pm}%
^{-1})^{2}\left(  s\right)  (p_{\pm})^{\prime}\left(  p_{\pm}^{-1}\left(s\right)\right),
\\
\dfrac{1}{\left[  \rho^{2}c_{w}^{2}\right]  \left(  y,s\right)  }=\dfrac
{y}{\left[  \rho_{+}^{2}c_{+}^{2}\right]  \left(  s\right)  }+\dfrac
{1-y}{\left[  \rho_{-}^{2}c_{-}^{2}\right]  \left(  s\right)  }.
\end{array}
\right.
\]
Observe that if $\rho_{+},\rho_{-}\in\left(  0,\infty\right)  $
are such that
\[
p_{+}\left(  \rho_{+}\right)  =p_{-}\left(  \rho_{-}\right)  =p
\]
then
\begin{equation}
\left\{
\begin{array}
[c]{l}%
\left[  \rho_{\pm}^{2}c_{\pm}^{2}\right]  \left(  p\right)  =\rho_{\pm}%
^{2}c_{\pm}^{2},\\
\dfrac{1}{\left[  \rho^{2}c_{w}^{2}\right]  \left(  y,p\right)  }=\dfrac
{y}{\rho_{+}^{2}p_{+}^{\prime}(\rho_{+})}+\dfrac{1-y}{\rho_{-}^{2}%
p_{-}^{\prime}\left(  \rho_{-}\right)  }.
\end{array}
\right.
\label{functii_ciudate}
\end{equation}
Consider the application
\[
\left(  \alpha_{+},\rho_{+},\rho_{-}\right)  \rightarrow\left(  y,\alpha
_{+}p_{+}\left(  \rho_{+}\right)  +\alpha_{-}p_{-}(\rho_{-}),{\pi}\left(
\rho,y\right)  \right)
\]
where $\rho$ and $y$ are seen as functions of $\left(  \alpha
_{+},\rho_{+},\rho_{-}\right)  $ via 
\[
\rho=\alpha_{+}\rho_{+}+(1-\alpha_{+})\rho_{-},\text{ }y=\frac{\alpha_{+}\rho_{+}%
}{\alpha_{+}\rho_{+}+(1-\alpha_{+})\rho_{-}}.
\]
The differential of this transformation computed at a point $\left(
\alpha_{+},\rho_{+},\rho_{-}\right)  $ where $p_{+}\left(  \rho_{+}\right)
=p_{-}\left(  \rho_{-}\right)  $ is
\[%
\begin{pmatrix}
\frac{\rho_{+}\rho_{-}}{\rho^{2}} & \frac{\alpha_{+}\alpha_{-}\rho_{-}}%
{\rho^{2}} & -\frac{\alpha_{+}\alpha_{-}\rho_{+}}{\rho^{2}}\\
0 & \alpha_{+}p_{+}^{{\prime}}\left(  \rho_{+}\right)   & \alpha_{-}%
p_{-}^{\prime}\left(  \rho_{-}\right)  \\
0 & \frac{\frac{\alpha_{+}}{\rho_{+}}}{\frac{\alpha_{+}}{\rho_{+}%
p_{+}^{{\prime}}\left(  \rho_{+}\right)  }+\frac{\alpha_{-}}{\rho_{-}%
p_{-}^{\prime}\left(  \rho_{-}\right)  }} & \frac{\frac{\alpha_{-}}{\rho_{-}}%
}{\frac{\alpha_{+}}{\rho_{+}p_{+}^{^{\prime}}\left(  \rho_{+}\right)  }%
+\frac{\alpha_{-}}{\rho_{-}p_{-}^{\prime}\left(  \rho_{-}\right)  }}%
\end{pmatrix}
\]
and the Jacobian is given by
\begin{equation}
J\left(  \alpha_{+},\rho_{+},\rho_{-}\right)  =\frac{\rho_{-}p_{-}^{\prime
}\left(  \rho_{-}\right)  \rho_{+}p_{+}^{{\prime}}\left(  \rho_{+}\right)
}{\alpha_{+}\rho_{-}p_{-}^{\prime}\left(  \rho_{-}\right)  +\alpha_{-}\rho
_{+}p_{+}^{{\prime}}\left(  \rho_{+}\right)  }\frac{\alpha_{+}\alpha_{-}%
}{\rho^{2}}\left(  \rho_{+}p_{+}^{{\prime}}\left(  \rho_{+}\right)  -\rho
_{-}p_{-}^{\prime}\left(  \rho_{-}\right)  \right)  .\label{jacobian_of}%
\end{equation}
Consider the change of variable
\begin{equation}
\varphi:\left\{  \left(  \alpha,p,\delta p\right)  \in\left(  0,1\right)
\times\left(  0,\infty\right)  \times\mathbb{R}:\dfrac{p}{\alpha}>\delta p>-\dfrac{p}{1-\alpha}\right\}
\rightarrow\mathbb{R}^{3}\label{def_varphi}%
\end{equation}
given by
\begin{equation}
\left\{
\begin{array}
[c]{l}%
\varphi_{1}\left(  \alpha_{+},p,\delta p\right)  =\dfrac{\alpha_{+}%
p_{+}^{-1}\left(  p_{+}\right)  }{\alpha_{+}p_{+}^{-1}\left(  p_{+}\right)
+\alpha_{-}p_{-}^{-1}\left(  p_{-}\right)  },\\
\varphi_{2}\left(  \alpha_{+},p,\delta p\right)  =p,\\
\varphi_{3}\left(  \alpha_{+},p,\delta p\right)  =p
-  {\pi}  \left(  \dfrac{\alpha_{+}p_{+}%
^{-1}\left(  p_{+}\right)  }{\alpha_{+}p_{+}^{-1}\left(  p_{+}\right)
+\alpha_{-}p_{-}^{-1}\left(  p_{-}\right)  },\alpha_{+}p_{+}^{-1}\left(
p_{+}\right)  +\alpha_{-}p_{-}^{-1}\left(  p_{-}\right)  \right)
\end{array}
\right.  \label{change_of_var}%
\end{equation}
with%
\[
\alpha_-=1-\alpha_+,\text{   }p_{\pm}=p\pm\alpha_\mp\delta p.
\]
Let
\[
\psi:\left(  0,1\right)  \times\left(  0,\infty\right)  \mathcal{\rightarrow
}\left(  0,1\right)  \times\left(  0,\infty\right)  :\psi\left(  \alpha
_{+},p\right)  =\left(  \varphi_{1}\left(  \alpha_{+},p,0\right)  ,\varphi
_{2}\left(  \alpha_{+},p,0\right)  \right).
\]
The function $\psi$ is a $C^{\infty}$-diffeomorphism which can be computed
explicitly. Using the chain rule and $\left(  \text{\ref{jacobian_of}%
}\right)  $, we see that the Jacobian of $\varphi$ is different from
zero at $\left(  \alpha,p,0\right)$ for
\[
\left(  \alpha,p\right)  \in\mathcal{O}:=\left\{  \left(  \alpha,s\right)
\in\left(  0,1\right)  \times\left(  0,\infty\right)  :p_{+}^{-1}%
(s)p_{+}^{\prime}\left(  p_{+}^{-1}(s)\right)  \not =p_{-}^{-1}(s)p_{-}%
^{\prime}\left(  p_{-}^{-1}(s)\right)  \right\}  .
\]
The following proposition shows that the change of
variable $\left(  \text{\ref{change_of_var}}\right)  $ is reversible on a
sufficiently small cylindrical neighborhood of $p_{+}=p_{-}$.

\begin{proposition}
Let $\mathcal{U}$ an open bounded convex set such that ${\overline
{\mathcal{U}}}\subset\mathcal{O}$ and $\varphi$ as defined in \eqref{def_varphi}-\eqref{change_of_var}. There exists some $\eta>0$ such that
the restriction of $\varphi$ to
${\mathcal{U}}\times\left(  -\eta,\eta\right)  $ is a $C^{\infty}$-diffeomorphism onto its image.
\end{proposition}

\bigskip

In order to obtain an equation for $\pi$, differentiating
$\left(  \text{\ref{relatie}}\right)$, we obtain
\begin{align*}
\frac{\partial_{x}u}{\rho} &  =-\left(\frac{y_{+}}{(p_{+}^{-1})^{2}\left(
{\pi}\right)  }(p_{+}^{-1})^{\prime}\left(  {\pi}\right)  +\frac{y_{-}%
}{(p_{-}^{-1})^{2}\left(  {\pi}\right)  }(p_{-}^{-1})^{\prime}\left(
{\pi}\right)  \right)D_{t}{\pi}\\
&  =-\left(\frac{y_{+}}{(p_{+}^{-1})^{2}\left(  {\pi}\right)  }\frac{1}%
{(p_{+})^{\prime}\left(  p_{+}^{-1}({\pi}\right)  )}+\frac{y_{-}}%
{(p_{-}^{-1})^{2}\left(  {\pi}\right)  }\frac{1}{(p_{-})^{\prime}\left(
p_{-}^{-1}({\pi}\right)  )}\right)D_{t}{\pi}.
\end{align*}
This leads to
\[
D_{t}{\pi}+\frac{1}{\rho}\left[  \rho^{2}c_{w}^{2}\right]  \left(
y,{\pi}\right)  \partial_{x}u=0.
\]
Let the other variable be
\[
\delta \pi:=p-{\pi}.
\]
The fact that $\delta {\pi}$ behaves as $\delta p=p_{+}-p_{-}$ can be
obtained by comparing $\left(  \text{\ref{relatie}}\right)  $ with%
\[
\frac{1}{\rho}=\frac{y_{+}}{\rho_{+}}+\frac{y_{-}}{\rho_{-}}=\frac{y_{+}%
}{p_{+}^{-1}(p_{+})}+\frac{y_{-}}{p_{-}^{-1}(p_{-})}.
\]
Using \eqref{functii_ciudate}, this gives%
\[
y_{+}\int_{{\pi}}^{p_{+}}\frac{ds}{\left[  \rho_{+}^{2}c_{+}^{2}\right]
\left(  s\right)  }+y_{-}\int_{{\pi}}^{p_{-}}\frac{ds}{\left[  \rho_{-}%
^{2}c_{-}^{2}\right]  \left(  s\right)  }=0.
\]
We deduce that%
\begin{align*}
& y_{+}\left(  p_{+}-p\right)  \int_{0}^{1}\frac{d\theta}{\left[  \rho_{+}%
^{2}c_{+}^{2}\right]  \left(  (1-\theta)p+\theta p_{+}\right)  }+y_{-}\left(
p_{-}-p\right)  \int_{0}^{1}\frac{d\theta}{\left[  \rho_{-}^{2}c_{-}%
^{2}\right]  \left(  (1-\theta)p+\theta p_{-}\right)  }\\
& =y_{+}\left(  {\pi}-p\right)  \int_{0}^{1}\frac{d\theta}{\left[  \rho
_{+}^{2}c_{+}^{2}\right]  \left(  (1-\theta)p+\theta {\pi}\right)  }%
+y_{-}\left(  {\pi}-p\right)  \int_{0}^{1}\frac{d\theta}{\left[  \rho
_{-}^{2}c_{-}^{2}\right]  \left(  (1-\theta)p+\theta {\pi}\right)  }.
\end{align*}
We obtain
\begin{align}
& \left( \rho_+ \int_{0}^{1}\frac{d\theta}{\left[  \rho_{+}^{2}c_{+}^{2}\right]
\left(  (1-\theta)p+\theta p_{+}\right)  }-\rho_-\int_{0}^{1}\frac{d\theta}{\left[
\rho_{-}^{2}c_{-}^{2}\right]  \left(  (1-\theta)p+\theta p_{-}\right)
}\right)  \frac{\alpha_{+}\alpha_{-}\delta p}{\rho}\nonumber\\
& =\left(  y_{+}\int_{0}^{1}\frac{d\theta}{\left[  \rho_{+}^{2}c_{+}%
^{2}\right]  \left(  (1-\theta)p+\theta {\pi}\right)  }+y_{-}\int_{0}%
^{1}\frac{d\theta}{\left[  \rho_{-}^{2}c_{-}^{2}\right]  \left(
(1-\theta)p+\theta {\pi}\right)  }\right)  \left(  {\pi}-p\right)
.\label{link_deltapeff_deltap}%
\end{align}
A second-order expansion of the left-hand side of the above equation reads

\begin{equation}
\mathrm{LHS}_{\text{\eqref{link_deltapeff_deltap}}}=\left(  \frac{1}{\rho
_{+}c_{+}^{2}}-\frac{1}{\rho_{-}c_{-}^{2}}\right)  \frac{\alpha
_{+}\alpha_{-}\delta p}{\rho}+O\left(  (\delta p)^{2}\right)  .
\label{link2}
\end{equation}
Gathering the equations we obtain%
\begin{align} \label{eq:sym1}
\left\{
\begin{array}
[c]{l}%
D_{t}y_{+}=0,\\
D_{t}{\pi}+\frac{1}{\rho}\left[  \rho^{2}c_{w}^{2}\right]  \left(
y_{+},{\pi}\right)  \partial_{x}u=0,\\
D_{t}\left(  p-{\pi}\right)  +\dfrac{1}{\rho}\left(  y_{+}\rho^{2}c_{+}%
^{2}+y_{-}\rho^{2}c_{-}^{2}-\left[  \rho^{2}c_{w}^{2}\right]  \right)
\partial_{x}u\\
\hspace{2.4cm}+\left(  \rho
_{+}c_{+}^{2}-\rho_{-}c_{-}^{2}\right)  \frac{\alpha_{+}\alpha_{-}}{\mu}\delta
p=\dfrac{\alpha_{+}\alpha_{-}}{\mu}(\delta
p)^2,\\
D_{t}u+\dfrac{1}{\rho}\partial_{x}p=0.
\end{array}
\right.
\end{align}
Using the equality \eqref{link_deltapeff_deltap} with \eqref{link2}, we can rewrite the last two terms appearing in the third equation of \eqref{eq:sym1} as a term proportional to
$\dfrac{(\delta {\pi})^{2}}{\mu}$ thus achieving the form announced in
$\left(  \text{\ref{form}}\right)  $.

\section{Main results} \label{sec:mainresults}
\subsection{Uniform-in-$\mu$ local-in-time well-posedness and pressure-relaxation limit}
% Let $\mathcal{O}$ be the constrained state space
% \[
% \mathcal{O}=\left\{  \alpha_{+}\in\left(  0,1\right),\:s>0,\:p_{+}^{-1}%
% (s)p_{+}^{\prime}\left(  p_{+}^{-1}(s)\right)  \not =p_{-}^{-1}(s)p_{-}%
% ^{\prime}\left(  p_{-}^{-1}(s)\right)  \right\}  \subset\mathbb{R}^{2}.%
% \]
\begin{theorem}\label{thm:LWP}
Let $\mu>0$. Let $K\subset\mathcal{O}$ a compact convex set, $\left(  \bar\alpha_+%
,\overline{p}\right)  \in\mathcal{O}$, $\bar\alpha_{-}=1-\bar
{\alpha}_{+}$ and, for $s>3/2$,
\[
\left(  \alpha_{+,0},p_{0},u_{0}\right)  \in\left(  \bar\alpha_+%
,\overline{p},0\right)  +\left(  H^{s}\left(  \mathbb{R}\right)  \right)
^{3}%
\]
such that $\left(  \alpha_{+,0},p_{0}\right)  \in K$. There exists $\eta=\eta\eqref{Kapila_chap_sym}>0$ such that, for any $\delta p_{0}\in H^{s}(\mathbb{R})$, if
\[
\|\delta p_{0}\|_{H^{s}} \leq \eta,
\]
then there exists $T=T\left(  \left\Vert \left(  \alpha_{+,0} - \bar \alpha_+,p_{0} - \bar p%
,u_{0}\right)  \right\Vert _{H^{s}},K\right)  $ independent of $\mu$ such that the \eqref{DBN} system with initial data $\left(  \alpha_{+,0},\rho_{\pm,0}%
,u_{0}\right)  $ where
\[
\left\{
\begin{array}
[c]{c}%
\rho_{+,0}=p_{+}^{-1}\left(  p_{0}+\alpha_{-,0}\delta p_{0}\right)  ,\\
\rho_{-,0}=p_{-}^{-1}\left(  p_{0}-\alpha_{+,0}\delta p_{0}\right)  ,
\end{array}
\right.  \text{ }%
\]
admits a unique solution $\left(  \alpha_{+}^{\mu},\rho_{\pm}^{\mu},u^{\mu
}\right)  \in C([0,T];H^{s}(\mathbb{R}))\cap C^{1}((0,T);H^{s-1}(\mathbb{R}%
))$. 
%Moreover, such that%
%\[
%\left\Vert \left(  \alpha_{+}^{\mu},\rho_{\pm}^{\mu},u^{\mu}\right)
%\right\Vert _{H^{s}}...
%\]

\end{theorem}

\begin{remark}
    The non-resonance condition \begin{align}\label{cond:p}
        p_{+}^{-1}%
(s)p_{+}^{\prime}\left(  p_{+}^{-1}(s)\right)  \not =p_{-}^{-1}(s)p_{-}%
^{\prime}\left(  p_{-}^{-1}(s)\right)
\end{align}
appearing in $\mathcal{O}$ is required in our analysis to justify the validity the change of variables \eqref{def_varphi}. In the case of gamma-law pressure $p_\pm(\rho_\pm)=c_\pm\rho_\pm^{\gamma_\pm}$, it is equivalent to $\gamma_+\ne \gamma_-$.
\end{remark}

Next, we show that the solution constructed in Theorem \ref{thm:LWP} converges strongly toward the solution of Kapila model \eqref{Kapila_chap_sym} for ill-prepared initial data.

\begin{theorem}\label{thm:relax} Let $\mu>0$.
    Let $K\subset {\mathcal O}$ a compact convex set and $\bar\alpha_+,\overline{p}\in {\mathcal O}$. Let $V_0^\mu=(\alpha_{+,0}^\mu,p_{+,0}^\mu,\delta p^\mu_0,u_0^\mu)\in (\bar\alpha_+,\overline{p},0,0)+(H^s(\R))^4$ satisfying the conditions of Theorem \ref{thm:LWP} and let $V_0 =(\alpha_{+,0},p_0,u_0)\in (\bar\alpha_+,\bar p_+,0) + (H^s(\R))^3$.
    Let $(\alpha_+^\mu,p_+^\mu,\delta p^\mu, u^\mu)$ (resp. $(\alpha_+,p_+,u)$) be the solution of \eqref{DBN} (resp. \eqref{Kapila_chap_sym}) with initial condition $(\alpha_{+,0}^\mu, p_{+,0}^\mu,\delta p_0^\mu, u_0^\mu)$ (resp. $(\alpha_{+,0},p_0,u_0)$). Then there exists $C>0$ depending only on $\Vert (V_0^\mu, V_0)\Vert_{H^s}$ such that
    \begin{equation*}
        \Vert(y^\mu-y, \pi^\mu-p,u^\mu-u)\Vert_{L^\infty_T(H^{s-1})} \leq C\sqrt{\mu} + C\Vert (\alpha_{0,+}^\mu-\alpha_{0,+}, p_{0,+}^\mu-p_{0,+},\delta p_{0}^\mu,u_0^\mu-u_0)\Vert_{H^{s-1}}
    \end{equation*}
    where $p$ is given by \eqref{Kapila_sym} and $\pi^\mu$ by \eqref{relatie}.
\end{theorem}

The main difficulty, addressed in Section \ref{sec:sym}, in proving Theorem \ref{thm:LWP} and Theorem \ref{thm:relax} lies in justifying that the \eqref{DBN} model can be rewritten in the general form \eqref{eq:sym1}. Note that, in order to derive a convergence rate in an ill-prepared setting as in Theorem~\ref{thm:relax}, it is crucial to have a symmetrization of the form \eqref{form} with $f_{22}=0$. We refer to \cite{GiovangigliYong2015} for a discussion on a symmetrization procedure, based on entropy variables, yielding \textit{normal forms} such as \eqref{eq:sym1} for other models.

The technical steps required to ensure existence and uniform bounds in Sobolev spaces are standard (see \cite{GiovangigliYong2018,BurteaCrin-BaratTan2023}) and are therefore omitted.

\subsection{Non-uniform-in-$\mu$ global-in-time well-posedness result}

Next, we wish to investigate the large time behavior of the system.
Define the mass-Lagrangian coordinates:
\[
\left\{
\begin{array}
[c]{l}%
X:[0,\infty)\times\mathbb{R\rightarrow R},\\
\dot{X}\left(  t,x\right)  =u\left(  t,X\left(  t,x\right)  \right)  ,\text{
}X\left(  0,x\right)  =x
\end{array}
\right.
\]
and
\[
Z:\mathbb{R\rightarrow R},\text{ }Z^{-1}\left(  x\right)  =\int_{0}^{x}%
\rho_{0}\left(  y\right)  \mathrm{d}y.
\]
System \eqref{DBN} for
\[
\left(  \alpha^\ell_{\pm},{\rho}^\ell_{\pm},u^\ell\right)  \left(
t,x\right)  =\left(  \alpha_{\pm},\rho_{\pm},u\right)  \left(
t,X(t,Z(x))\right)
\]
reads%
\begin{equation}
\left\{
\begin{array}
[c]{l}%
\partial_{t}\alpha^\ell_{+}=\dfrac{\alpha^\ell_{+}\alpha^\ell_{-}}%
{\mu}(p^\ell_{+}-p^\ell_{-}),\\
\partial_{t}({\alpha}^\ell_{\pm}{\rho}^\ell_{\pm})+{\alpha}^\ell_{\pm
}{\rho}^\ell_{\pm}{\rho}^\ell\partial_{x}u^\ell=0,\\
\partial_{t}u^\ell+\partial_{x}({\alpha}^\ell_{+}{p}^\ell_{+}%
+{\alpha}^\ell_{-}{p}^\ell_{-})=0.
\end{array}
\right.  \label{mass_lagrangian_BN}%
\end{equation}
In these coordinates, we also have that%
\begin{align}\label{eq:yLag}
\partial_{t}y^\ell_{+}=0.
\end{align}
% In the following, we drop the tilde-notation. 

% We recall that that $\delta p$ and $\delta\pi$ are linked by relation $\left(
% \text{see which one}\right)  $.
Our global-in-time well-posedness result reads as follows.
\begin{theorem} \label{thm:GWP}
Let $\mu>0$. For $\bar\alpha_\pm,\bar\rho_\pm>0$, there exists a constant $\eta>0$ depending on $\mu$ such that for $s>3/2$, for any $\alpha_{\pm,0}, \rho_{\pm,0},u_0\in (\bar\alpha_+,\overline{\rho}_\pm,0)+(H^s(\R))^3$, if
   \[
\left\Vert (\alpha_{\pm,0}-\bar{\alpha}_{\pm}, \rho_{\pm,0}-\bar{\rho}_{\pm}, u_0)\right\Vert _{H^s}\leq \eta,
\]
    then \eqref{mass_lagrangian_BN} admits a unique global-in-time solution which satisfies for all $T>0$ : 
    $$Y(T)\leq 2\eta,$$
    where
\begin{align}\label{Xcontrol}
Y(T)&=\|(\alpha^\ell_{\pm}-\bar\alpha_\pm,\rho^\ell_{\pm}-\bar\rho_\pm, u^\ell)\|_{L^\infty_T(H^s)}+\|\delta \pi^\ell\|_{L^2_T(H^s)}+\|(\partial_x u^\ell,\partial_x{\pi}^\ell)\|_{L^2_T(H^{s-1})}.
\end{align}
\end{theorem}

\begin{remark} \hfill
    \begin{itemize}
     \item Theorem \ref{thm:GWP} highlights the hypocoercive structure of the model \eqref{mass_lagrangian_BN}: although dissipative effects appears explicitly only in the pressure difference, the hyperbolic coupling within the system allows this dissipation to propagate, yielding decay for both the effective pressure $\pi$ and the velocity.
        \item While Theorems \ref{thm:LWP} and \ref{thm:relax} can be readily extended to multi-dimensional settings (upcoming work), Theorem \ref{thm:GWP} is restricted to the one-dimensional case. This limitation stems from the fact that the linearization of the sub-system $\eqref{eq:sym1}_{1-3}$ does not satisfy the Shizuta–Kawashima (SK) condition \cite{Shizuta1985}  in dimensions greater than one. Computations regarding this fact are provided in Appendix \ref{app:noSK}.
      \item Theorem \ref{thm:GWP} provides a global-in-time existence result for the mass--Lagrangian formulation of the \eqref{DBN} model. Although a corresponding result can be derived for the Eulerian formulation, the control on the solution would be weaker: for initial data in \(H^s\), the solution can only be controlled uniformly-in-time in \(H^{s-1}\). This is due to the fact that we are not able to recover an \(L^1_T(W^{1,\infty})\)-bound on the velocity, which is required to either justify the mass-Lagrangian change of frame or to control the advection term in the equation of $y$ in the Eulerian frame.
      
      % To remedy this issue, a natural strategy would be to work within a Besov framework in order to recover such time-integrable Lipschitz control, as in \cite{BurteaCrin-BaratTan2023}. However, the space \(B^{-\frac12}_{2,1} \cap B^{\frac 32}_{2,1}\), which is typically suited for this purpose, turns out to be inadequate in dimension one, since the product laws required to close the analysis fail in this setting.
        \item In a setting closely related to the one considered here, Qu and Wang \cite{QuWangNoSk} proved a global existence result for quasilinear hyperbolic systems in which exactly one characteristic family violates the Shizuta--Kawashima (SK) condition, provided that the nonlinear terms satisfy a suitable degeneracy condition with respect to this family. In our framework, the eigen-family failing the (SK) condition corresponds to the variable $y$, and the nonlinearities exhibit a similar degeneracy structure. However, the results of \cite{QuWangNoSk} cannot be directly applied to our system. First, 
        their approach relies on the existence of a strictly convex entropy, an assumption that is not fulfilled in our setting. Secondly, their analysis is restricted to space dimensions $d \geq 2$, where the Lipschitz bound \(L^1_T(W^{1,\infty})\) on $u$ can be recovered by assuming that the initial data is in $L^1$ (something that would not work in a one-dimensional setting, as it would still not yield sufficient decay).
    \end{itemize}
\end{remark}

% \begin{theorem} \label{thm:GWPBN}
%     There exists a constant $\eta>0$ depending on $\mu$ such that if
%     \begin{align}
%         \label{smalldata}
%         X_0=\|(\dfrac{\alpha_{+0}\rho_{+0}}{\rho_0}-\dfrac{\bar\alpha_{+0}\bar\rho_{+0}}{\bar\rho_0},\frac{\rho_{+0} c_{+0}^{2} - \rho_{-0} c_{-0}^{2}}
%      {\dfrac{\rho_{+0} c_{+0}^{2}}{\alpha_{+0}}
%       + \dfrac{\rho_{-0} c_{-0}^{2}}{\alpha_{-0}}}
% \,(P_{+0} - P_{-0}),\left(\frac{\dfrac{\rho_{-0} c_{-0}^{2}}{\alpha_{-0}}}{\dfrac{\rho_{+0} c_{+0}^{2}}{\alpha_{+0}} + \dfrac{\rho_{-0} c_{-0}^{2}}{\alpha_{-0}}}\, P_{+0}
% \;+\;
% \frac{\dfrac{\rho_{+0} c_{+0}^{2}}{\alpha_{+0}}}{\dfrac{\rho_{+0} c_{+0}^{2}}{\alpha_{+0}} + \dfrac{\rho_{-0} c_{-0}^{2}}{\alpha_{-0}}}\, P_{-0}\right)-\bar P_0,u_0)\|_{H^s}\leq \eta,
%     \end{align}
%     then \eqref{DBN} admits a unique global-in-time solution which, for all $t>0$, satisfies
%     $$X(t)\leq X_0,$$
%     where
% \begin{align}\label{Xcontrol}
% X(T)&=\|y\|_{L^\infty_T(H^s)}+\|(\delta \pi,{\pi},u)\|_{L^\infty_T(H^s)}\\&\quad+\|w\|_{L^2_T(L^2)}+\|(\partial_xu,\partial_x{\pi})\|_{L^2_T(L^2)}+\|(\delta \pi,{\pi},u)\|_{L^2_T(\dot{H}^s)}. \nonumber
% \end{align}
% \end{theorem}

% \section{Proof of Theorem \ref{thm:LWP} and Theorem \ref{thm:relax}}

\section{Proof of Theorem \ref{thm:GWP}}\label{sec:gwp}

\subsection{Linearization}
In this section, we drop the $\ell$-notation related to the mass-Lagrangian framework. System \eqref{mass_lagrangian_BN} can also be rewritten with respect to the $\left(  y,\pi,\delta
\pi,u\right)$ variables as
\begin{equation}
\left\{
\begin{array}
[c]{l}%
\partial_{t}y=0,\\
\partial_{t}\delta\pi+\left(  y\rho^{2}c_{+}^{2}+\left(  1-y\right)  \rho
^{2}c_{-}^{2}-\left[  \rho^{2}c_{w}^{2}\right]  \right)  \partial_{x}u+\left(
\rho_{+}^{2}c_{+}^{2}-\rho_{-}^{2}c_{-}^{2}\right)  \frac{\alpha_{+}\alpha
_{-}}{\mu}\delta p=0,\\
\partial_{t}\pi+\left[  \rho^{2}c_{w}^{2}\right]  \left(  y,\pi\right)
\partial_{x}u=0,\\
\partial_{t}u+\partial_{x}p=0.
\end{array}
\right.  \label{mass_lag_symetric}%
\end{equation}
Linearizing \eqref{mass_lag_symetric} around the equilibrium 
\begin{align}
    \label{eqeq}(\bar{y},0,\bar{P},0)
\end{align}
we obtain
\begin{equation}
\left\{
\begin{array}
[c]{l}%
\partial_{t}\widetilde y=0,\\
\partial_t\delta \pi+\bigl(h_1+H_{1}%
(\delta \pi,\widetilde{\pi}, \widetilde y)\bigr)\partial_xu+\bigl(h_2+H_{2}(\delta \pi,\widetilde{\pi}, \widetilde y)\bigr)\delta\pi=0,\\
\partial_{t}{\widetilde{\pi}}+\bigl(h_3+H_{3}%
(\delta \pi,\widetilde{\pi}, \widetilde y)\bigr)\partial_xu=0,\\
\partial_{t}u+\bigl(h_5+H_{5}(\delta \pi,\widetilde{\pi}, \widetilde y)\bigr)\partial_x {\widetilde{\pi}}+\bigl(h_6+H_{6}(\delta \pi,\widetilde{\pi}, \widetilde y)\bigr)\partial_x \delta \pi=0,
\end{array}
\right.
\label{generalform}
\end{equation} 
where the perturbed unknown are defined as follows:
\begin{itemize}
\item the perturbed mass fraction $y$ verifies
\begin{align}
\widetilde{y}:=\dfrac{\alpha_+\rho_+}{\rho}-\dfrac{\bar\alpha_+\bar\rho_+}{\bar\rho}
\end{align}
    \item the perturbed effective pressure $ \widetilde\pi$ reads
    \[
\widetilde{\pi}
:= \pi-\bar{P}
\]
\item For $i\in\overline{1,5}$, the functions $H_i$ satisfy $H_{i}\left(  0,0,0\right)=0$ and the constants $h_i>0$ are positive.
\end{itemize}

% As we consider strong solutions and the velocity will be at least Lipschitz, system \eqref{generalform}is equivalent to its Lagrangian formulation.

\subsection{A priori estimates}
%In the following, we omit the tilde-notation  related to the perturbated unknowns. Consider
\begin{equation}
X^{2}\left(  t\right)  :=\left\Vert \left(  \delta\pi,\widetilde{\pi},u\right)  \left(
t\right)  \right\Vert _{H^{s}(\mathbb{R})}^{2}+\left\Vert \delta\pi\right\Vert
_{L^{2}(\left(  0,t\right)  ;H^{s}(\mathbb{R}))}^{2}+\left\Vert (\widetilde{\pi}
,u)\right\Vert _{L^{2}(\left(  0,t\right)  ;\dot{H}^{1}(\mathbb{R})\cap
\dot{H}^{s}(\mathbb{R}))}^{2}.\label{definitie_X(t)}%
\end{equation}
In the following we assume that
\[
X\left(  0\right)  =\left\Vert \left(  \delta\pi_{0},\widetilde{\pi}_{0},u_{0}\right)
\right\Vert _{H^{s}(\mathbb{R})}\leq\eta
\]
and consider the time $T^{\ast}\in(0,\infty]$ where
\begin{align}
    \label{smallnessU}
X\left(  t\right)  \leq2\eta
\end{align}
Assume also that
\[
\left\Vert \widetilde{y}_{0}\right\Vert _{H^{s}}\leq c
\]
with $c$ to be chosen small. We have
\begin{align}\label{est:y}
  \|\widetilde{y}\|_{L^\infty_T(H^s)} \leq c.
\end{align}
\subsubsection{Energy estimates}
For the unknown $(\widetilde y,\delta \pi,\widetilde\pi,u)$, we have the following lemma.
\begin{lemma}\label{lem:hs} Let $s>3/2$ and $(\delta \pi,\widetilde{\pi},u)$ be a smooth solution of \eqref{generalform} such that \eqref{smallnessU} holds.  For all times $t>0$, $$ \|(\widetilde y,\delta \pi,\widetilde{\pi},u)\|_{H^s}^2+c\|\delta \pi\|_{L^2_T(H^s)}^2 \leq \|(\widetilde y_0,\delta \pi_0,\widetilde{\pi}_0,u_0)\|_{H^s}^2+X^3(t).$$
\end{lemma}
\begin{proof}
For the $L^2$ estimates, we shall rely on the natural energy associated to System\eqref{mass_lagrangian_BN}.
System $\left(
\text{\ref{mass_lagrangian_BN}}\right)$ comes with an energy which satisfies
\begin{equation}
\frac{d}{dt}\int_{\mathbb{R}}\left(\frac{\left\vert u\right\vert ^{2}}%
{2}+y_{+,0}\frac{H_{+}\left(  \rho_{+}/\bar{\rho}_{+}\right)  }%
{\rho_{+}}+y_{-,0}\frac{H_{-}\left(  \rho_{-}/\bar{\rho}%
_{-}\right)  }{\rho_{-}}\right)+\int_{\mathbb{R}}\frac{\alpha
_{+}\alpha_{-}}{\mu\rho}(p_{+}-p_{-})^{2}=0\label{energy}%
\end{equation}
where%
\[
H_{\pm}\left(  s/\bar{\rho}_{\pm}\right)  =H_{\pm}\left(  s\right)  -H\left(
\bar{\rho}_{\pm}\right)  -H^{\prime}\left(  \bar{\rho}_{\pm}\right)
(s-\bar{\rho}_{\pm}),
\]
and%
\[
H_{\pm}\left(  s\right)  =s\int_{0}^{s}\frac{p_{\pm}\left(  \tau\right)  }%
{\tau^{2}}\mathrm{d}\tau
\]
 is strictly convex since $p_\pm$ are both increasing functions.
In particular, on $[0,T^{\ast})$, using the energy
of the system $\left(  \text{\ref{energy}}\right)  $ provides a control for
\[
\left\Vert \left(  \widetilde{\rho}_{\pm},u\right)  \right\Vert _{L^{2}(\mathbb{R})}%
^{2}+\int_{0}^{t}\left\Vert \delta p\right\Vert _{L^{2}(\mathbb{R})}%
^{2}\lesssim X^{2}\left(  0\right)  +cX^{2}\left(  t\right)  +X^{3}\left(
t\right)  .
\]
Observing that%
\[
\alpha_{+}=\dfrac{\dfrac{\rho_{-}}{y_{-,0}}}{\dfrac{\rho_{+}}{y_{+,0}}+\dfrac
{\rho_{-}}{y_{-,0}}}%
\]
allows us to recover that%
\[
\left\Vert \left(  \widetilde{\alpha}_{+},\widetilde{\rho}_{\pm},u\right)  \right\Vert _{L^{2}%
(\mathbb{R})}^{2}+\int_{0}^{t}\left\Vert \delta p\right\Vert _{L^{2}%
(\mathbb{R})}^{2}\lesssim X^{2}\left(  0\right)  +X^{2}\left(  t\right)
+X^{3}\left(  t\right)  .
\]
This implies that%
\begin{equation}
\left\Vert \left(  \widetilde{\pi},\delta\pi,u\right)  \right\Vert _{L^{2}(\mathbb{R}%
)}^{2}+\int_{0}^{t}\left\Vert \delta\pi\right\Vert _{L^{2}(\mathbb{R})}%
^{2}\lesssim X^{2}\left(  0\right)  +c^{2}\left(  t\right)  +X^{3}\left(
t\right)  .\label{energy_L2}%
\end{equation}
For the $\dot H^{s}$ estimates, we define
$$
\Lambda^{\sigma}= (-\Delta)^\frac{\sigma}{2}=\mathcal{F}^{-1} \Big( |\xi|^{\sigma} \mathcal{F}(\cdot ) \Big),\quad\sigma\in\mathbb{R}.
$$ Applying $\Lambda^s$ to \eqref{mass_lagrangian_BN}, performing standard energy estimates and exploiting the symmetrizability of the system, together with standard commutator and product estimates, yields the desired result.
\end{proof}

\subsubsection{Hypocoercive estimates in $L^2$}
In Lemma \ref{lem:hs}, we recovered dissipative effects for the unknown $\delta \pi$.
We now aim to recover dissipation for the unknowns ${\pi}$ and $u$. To do so, we resort to hypocoercive estimates.

\begin{lemma}\label{lem:hypo1}  Let $s>3/2$ and $(\delta \pi,\widetilde\pi,u)$ be a smooth solution of \eqref{generalform} such that \eqref{smallnessU} holds. For $t>0$, we have
   \begin{align}\label{L2Hypo}
  \int_\R \delta \pi\partial_x u+\int_0^t\dfrac{h_1}{2}\|\partial_x u\|_{L^2}^2-\int_0^t\dfrac{h_2^2}{2h_1}\|\delta \pi\|_{L^2}^2&\lesssim \|(\delta \pi_0,\partial_xu_0)\|_{L^2}^2\\&\quad +\int_0^t\int_\R h_5\partial_x \widetilde{\pi} { \partial_x\delta \pi}+X^3(t) \nonumber
   \end{align}
   and 
     \begin{align}\label{HsHypo}
      \int_{\R}\Lambda^{s-1}\delta \pi\Lambda^{s-1}\partial_x u+\dfrac{h_1}{2}\int_0^t\|u\|_{\dot{H}^{s}}^2-\dfrac{h_2^2}{2h_1}\int_0^t\|\delta \pi\|_{\dot{H}^{s-1}}^2&\lesssim \|(\delta \pi_0,u_0)\|_{H^s}^2\\&\quad +\int_0^t\int_\R h_5\Lambda^{s-1}\partial_x \widetilde{\pi} { \Lambda^{s-1}\partial_x\delta \pi}\nonumber
      \\&\quad +X^3(t) \nonumber
    \end{align}
\end{lemma}
\begin{proof}
One has
\begin{align*}\dfrac{d}{dt}\int_\R \delta \pi\partial_x u +h_1\|\partial_x u\|_{L^2}^2- h_6\|\partial_x \delta \pi\|_{L^2}^2&=\int_\R H_6(\partial_x\delta \pi)^2 -\int_\R H_1(\partial_x u)^2\\&-\int_\R(h_2+H_2)\delta \pi\partial_xu\\&+\int_\R(h_5+H_5)\partial_x \widetilde{\pi} { \partial_x\delta \pi}.
\end{align*}
Employing Cauchy-Schwarz, Young inequalities and integrating in time give the left-hand side of \eqref{L2Hypo}.
% with 
% \begin{align}\label{c13}
%     c_1=h_1/2, \quad c_2=h^2_2 \quad \text{and} \quad c_3=h_5^2/h_1.
% \end{align}
The nonlinear terms can be treated using standard product estimates in Sobolev spaces. The proof of \eqref{HsHypo} is a direct generalization.
\end{proof}
Next, we show similar dissipative estimates for ${\pi}$.
\begin{lemma}\label{lem:hypo2}
Let $s>3/2$ and $(\delta \pi,\widetilde\pi,u)$ be a smooth solution of \eqref{generalform} such that \eqref{smallnessU} holds. For $t>0$, we have  $$ \int_\R u\partial_x \widetilde{\pi}+\dfrac{h_5}{2}\int_0^t\|\partial_x \widetilde{\pi}\|_{L^2}^2-h_3\int_0^t\|\partial_xu\|_{L^2}^2-\dfrac{h_6^2}{2h_5}\int_0^t\|\partial_x\delta \pi\|_{L^2}^2 \lesssim  \|(u_0,\partial_x\widetilde\pi_0)\|_{L^2}^2+X^3(t) $$
     and, for $s> d/2+1$, 
     \begin{align*}
      \int_{\R}\Lambda^{s-1}u\Lambda^{s-1}\partial_x \widetilde{\pi}+\dfrac{h_5}{2}\int_0^t\|\widetilde{\pi}\|_{\dot{H}^s}-h_3\int_0^t\|u\|_{\dot{H}^{s}}^2-\dfrac{h_6^2}{2h_5}\int_0^t\|\delta \pi\|_{\dot{H}^s}^2&\lesssim \|(u_0,\widetilde\pi_0)\|_{H^s}^2\\&\quad +X^3(t)
    \end{align*}
\end{lemma}
\begin{proof}
    The proof is similar to that of Lemma \ref{lem:hypo1}.
\end{proof}

% \subsection{Hypocoercive estimates in $\dot{H}^s$.}

% \begin{lemma}\label{lem:hypo3}
%    Let $(y,U)$ be a smooth solution of \eqref{generalform0} such that \eqref{smallnessU} holds.
%   There exist postive constants $c_4,c_5$ and $c_5$ depending on $h_1,h_,h_5$ and $h_6$ such that 
%     \begin{align}
%         \dfrac{d}{dt}\int_{\R} \Lambda^{s-1}w \Lambda^{s-1}\partial_xu+\|u\|^2_{\dot{H}^s}\lesssim \|(\delta \pi,\partial_xw)\|_{\dot{H}^{s-1}}^2+\|{\pi}\|_{\dot{H}^s}^2+X^3(t)
%     \end{align}
% \end{lemma}
% \begin{lemma}\label{lem:hypo4}
%     We have
%     \begin{align}
%        \dfrac{d}{dt}\int_{\R}\Lambda^{s-1}u\Lambda^{s-1}\partial_x {\pi}+\|{\pi}\|^2_{\dot{H}^s}\lesssim \|u\|_{\dot{H}^{s}}^2+\|w\|_{\dot{H}^s}^2+X^3(t)
%     \end{align}
% \end{lemma}

\subsubsection{Gathering the estimates}
Let $\varepsilon>0$. Consider the Lyapunov functional
\begin{align}\mathcal{L}(t)&=\|(\widetilde{y},\widetilde{\pi},\delta \pi,u)\|_{H^s}^2+\varepsilon\int_\R \delta\pi\partial_x u+\varepsilon^{\frac32}\int_\R u\partial_x \widetilde{\pi}
\\+&\varepsilon \int_\R \Lambda^{s-1}\delta\pi\Lambda^{s-1}\partial_x u+\varepsilon^{\frac32}\int_\R \Lambda^{s-1}u\Lambda^{s-1}\partial_x \widetilde{\pi}.
\end{align}
First, let us remark that for $\varepsilon$ sufficiently small, we have
\begin{align}
    \mathcal{L}(t)\sim \|(\widetilde{y},\widetilde{\pi},\delta \pi,u)\|_{H^s}^2.
\end{align}
Then, using that
\begin{equation*}
    \varepsilon h_5\vert\partial_x\widetilde\pi\vert \vert\partial_x\delta\pi\vert\leq \frac{\varepsilon^{3/2}h_5}{4}\vert\partial_x\widetilde\pi\vert^2 + \varepsilon^{1/2}h_5\vert\partial_x\delta\pi\vert^2
\end{equation*}
to treat the right-hand side of {\eqref{L2Hypo}}, 
together with Lemma \ref{lem:hs}, Lemma \ref{lem:hypo1}, Lemma \ref{lem:hypo2}, for $\varepsilon$ sufficiently small, we obtain
\begin{align}
    X(t)^2\lesssim X^3(t)+X_0^2.
\end{align}
Using a standard bootstrap argument extending our local-in-time result, one concludes the existence of a global-in-time solution satisfying the desired properties. The proof of uniqueness is omitted as it is classical. \qed

\subsection*{Acknowledgments}
This work was partially supported by a CNRS PEPS JCJC grant.
T. Crin-Barat is supported by the project ANR-24-CE40-3260 – Hyperbolic Equations, Approximations $\&$ Dynamics (HEAD).

\appendix

\section{On a higher-dimensional version of Theorem \ref{thm:GWP}}\label{app:noSK}
In this section, we explain why a higher-dimensional analogue of Theorem \ref{thm:GWP} cannot be obtained by directly replicating the methodology developed in Section \ref{sec:gwp}.

For $d\geq 2$, a simplified version of the linear system associated to the linearization of \eqref{eq:sym1} reads
\begin{align}
\left\{
\begin{array}
[c]{l}%
\partial_t \delta \pi + \div u+\delta\pi=0,
\\
\partial_t \pi + \div u=0,
\\
\partial_t u+\nabla \pi + \nabla \delta \pi=0.
\end{array}
\right.
\end{align}
Thanks to the symmetry of the system, standard $H^1$ energy estimates give
\begin{align}
    \dfrac12\dfrac{d}{dt}\|(\delta \pi, \pi,u)\|_{H^1}^2+\|\delta \pi\|_{H^1}^2=0.
\end{align}
To retrieve dissipation for $u$ and $\delta \pi$, we perform the following hypocoercive analysis.
We compute
\begin{align}
    \dfrac{d}{dt}\int_{\R^d}\delta \pi \cdot \div u + \|\div u\|_{L^2}^2-\|\nabla \delta \pi\|_{L^2}^2\leq \int_{\R^d} \delta\pi\cdot \div u+\int_{\R^d} \nabla\delta\pi\cdot \nabla \pi
\end{align}
and 
\begin{align}
    \dfrac{d}{dt}\int_{\R^d}u \nabla \pi + \|\nabla \pi\|_{L^2}^2-\|\div u\|_{L^2}^2\leq \int_{\R^d} -\nabla\delta\pi\cdot \nabla \pi.
\end{align}
Define. $$\mathcal{L}(t)= \dfrac12\dfrac{d}{dt}\|(\delta \pi, \pi,u)\|_{H^1}^2+\varepsilon\int_{\R^d}\delta \pi \cdot \div u+ \varepsilon^{\frac32}\int_{\R^d}u \nabla \pi.$$ 
For $\varepsilon$ small enough we have $\mathcal{L}\sim \|(\delta \pi, \pi,u)\|_{H^1}^2$. Employing Cauchy-Schwarz and Young inequalities yields
\begin{align}
    \dfrac{d}{dt}\mathcal{L}+\|\delta \pi\|_{H^1}^2+ \varepsilon\|\div u\|_{L^2}^2+\varepsilon^{\frac32}\|\nabla \pi\|_{L^2}^2 \lesssim 0.
\end{align}
We observe that we only retrieve dissipation for $\div u$, and not for the full gradient of the velocity. This structural limitation is due to the failure of the Shizuta--Kawashima (SK) condition in dimensions \(d \geq 2\). This failure is in turn related to the scalar nature of the pressure unknowns, in contrast with the vector-valued velocity field \(u\), which prevents the dissipation mechanism to be transferred over all velocity components.

\section{The full system}

Here, we present a symmetrization procedure for the full, non-isentropic, multi-dimensional Baer-Nunziato system in \cite{Kapila2001} which reads
\begin{equation}
\left\{
\begin{array}
[c]{l}%
\partial_{t}\alpha_{\pm}+u\cdot\nabla\alpha_{\pm}=\frac{\alpha_{\pm}\alpha
_{-}}{\mu}\left(  p_{+}-p_{-}-\beta\right),\\
\partial_{t}\left(  \alpha_{\pm}\rho_{\pm}\right)  +\operatorname{div}\left(
\alpha_{\pm}\rho_{\pm}u\right)  =\pm\mathcal{C},\\
\partial_{t}\left(  \rho u\right)  +\operatorname{div}\left(  \rho u\otimes
u\right)  +\nabla\left(  \alpha_{+}p_{+}+\alpha_{-}p_{-}\right)  =0.\\
\alpha_{+}\rho_{+}T_{+}\left(  \partial_{t}\eta_{+}+u\cdot\nabla\eta
_{+}\right)  =\frac{\alpha_{+}\alpha_{-}}{\mu}\left(  p_{+}-p_{-}%
-\beta\right)  ^{2}-\left(  T_{+}-T_{-}\right)  \mathcal{H},\\
\alpha_{-}\rho_{-}T_{-}\left(  \partial_{t}\eta_{-}+u\cdot\nabla\eta
_{-}\right)  =\left(  T_{+}-T_{-}\right)  \mathcal{H-}\left[  \left(
h_{+}-h_{-}\right)  -\frac{1}{\rho_{+}}\left(  p_{+}-p_{-}-\beta\right)
\right]  C,
\end{array}
\right.  \label{BN_complet}%
\end{equation}
where the extra unknowns are
\begin{itemize}
\item the temperatures: $T_{+},T_{-},$

\item the internal energies: $e_{+},e_{-},$

\item the specific entropies: $\eta_{+},\eta_{-},$

\item the enthalpies: $h_{+},h_{-},$

\item the function $\beta$ represents the intragranular stress or
configuration pressure:%
\[
\beta\left(  \alpha_{+},\rho_{+}\right)  =\alpha_{+}\rho_{+}\frac{dB}{d\alpha
}\left(  \alpha_{+}\right)
\]
where $B$ is a strictly increasing and strictly convex function. It is
important both from a physical point of view (\cite{Kapila2001}) and it somehow eases the
mathematical study of related systems (\cite{evje1})

\item the heat-transfer coefficients: $\mathcal{H}$.

\item mass-transfer coefficients: $\mathcal{C}$
\end{itemize}

We observe that the equation of the entropies have a favorable structure regarding the zero-order terms: the singular terms are proportional to $\dfrac{(\delta {\pi})^2}{\mu}$.

System \eqref{BN_complet} can be put under the form%
\[
\left\{
\begin{array}
[c]{l}%
D_{t}y=0,\\
D_{t}\delta {\pi}+\left(  \rho c^{2}-\rho c_{w}^{2}+f_{11}\delta
{\pi}\right)  \operatorname{div}u+\alpha_{+}\alpha_{-}\left(  \frac{\rho
_{+}c_{+}^{2}}{\alpha_{+}}+\frac{\rho_{-}c_{-}^{2}}{\alpha_{-}}+\alpha_{+}%
\rho_{+}B^{\prime\prime}\left(  \alpha_{+}\right)  +f_{12}\delta
{\pi}\right)  \dfrac{\delta {\pi}}{\mu}=0,\\
D_{t}{\pi}+(\rho c_{w}^{2}+f_{21}\delta {\pi})\operatorname{div}u+f_{22}%
\dfrac{\left(  \delta {\pi}\right)  ^{2}}{\mu}=0,\\
D_{t}u+\frac{1}{\rho}\nabla {\pi}+\frac{1}{\rho}\nabla(\delta {\pi})=0,\\
\alpha_{+}\rho_{+}T_{+}D_{t}\eta_{+}=\frac{\alpha_{+}\alpha_{-}}{\mu}\left(
p_{+}-p_{-}-\beta\right)  ^{2}-\left(  T_{+}-T_{-}\right)  \mathcal{H},\\
\alpha_{-}\rho_{-}T_{-}D_{t}\eta_{-}=\left(  T_{+}-T_{-}\right)
\mathcal{H-}\left[  \left(  h_{+}-h_{-}\right)  -\frac{1}{\rho_{+}}\left(
p_{+}-p_{-}-\beta\right)  \right]  C,
\end{array}
\right.
\]
where
\begin{align*}
\delta {\pi} &  =\frac{\rho_{+}c_{+}^{2}-\rho_{-}c_{-}^{2}-\beta}{\frac{\rho
_{+}c_{+}^{2}}{\alpha_{+}}+\frac{\rho_{-}c_{-}^{2}}{\alpha_{-}}+\alpha_{+}%
\rho_{+}B^{\prime\prime}\left(  \alpha_{+}\right)  }\left(  p_{+}-p_{-}%
-\beta\right)  ,\\
\rho c^{2} &  =\alpha_{+}\rho_{+}c_{+}^{2}+\alpha_{-}\rho_{-}c_{-}^{2},\\
\rho c_{w}^{2} &  =\rho c^{2}-\frac{\left(  \rho_{+}c_{+}^{2}-\rho_{-}%
c_{-}^{2}-\beta\right)  ^{2}}{\frac{\rho_{+}c_{+}^{2}}{\alpha_{+}}+\frac
{\rho_{-}c_{-}^{2}}{\alpha_{-}}+\alpha_{+}\rho_{+}B^{\prime\prime}\left(
\alpha_{+}\right)  },
\end{align*}
and%
\begin{align*}
{\pi}&=p-\frac{\rho_{+}c_{+}^{2}-\rho_{-}c_{-}^{2}-\beta}{\frac{\rho_{+}%
c_{+}^{2}}{\alpha_{+}}+\frac{\rho_{-}c_{-}^{2}}{\alpha_{-}}+\alpha_{+}\rho
_{+}B^{\prime\prime}\left(  \alpha_{+}\right)  }(p_{+}-p_{-}-\beta)\\
&=\frac{\frac{\rho_{-}c_{-}^{2}}{\alpha_{-}}+\alpha_{+}^{2}\rho_{+}%
B^{\prime\prime}\left(  \alpha_{+}\right)  +\beta}{\frac{\rho_{+}c_{+}^{2}%
}{\alpha_{+}}+\frac{\rho_{-}c_{-}^{2}}{\alpha_{-}}+\alpha_{+}\rho_{+}%
B^{\prime\prime}\left(  \alpha_{+}\right)  }p_{+}+\frac{\frac{\rho_{+}%
c_{+}^{2}}{\alpha_{+}}+\alpha_{-}\alpha_{+}\rho_{+}B^{\prime\prime}\left(
\alpha_{+}\right)  +\beta}{\frac{\rho_{+}c_{+}^{2}}{\alpha_{+}}+\frac{\rho
_{-}c_{-}^{2}}{\alpha_{-}}+\alpha_{+}\rho_{+}B^{\prime\prime}\left(
\alpha_{+}\right)  }p_{-}\\
&\quad+\frac{\rho_{+}c_{+}^{2}-\rho_{-}c_{-}^{2}-\beta}{\frac{\rho_{+}c_{+}^{2}%
}{\alpha_{+}}+\frac{\rho_{-}c_{-}^{2}}{\alpha_{-}}+\alpha_{+}\rho_{+}%
B^{\prime\prime}\left(  \alpha_{+}\right)  }\beta.
\end{align*}
Under this form, a local-in-time well-posedness result similar to Theorem \ref{thm:LWP} can be established for system
$\left(  \text{\ref{BN_complet}}\right) $.
\appendix
\bibliographystyle{abbrv} 
\bibliography{Reference}
\vfill

\end{document}